\documentclass[12pt]{amsart}

\usepackage{amsmath,amsthm,amssymb}
\usepackage{graphicx,psfrag,epsfig, color}

\numberwithin{equation}{section}
\theoremstyle{plain}
\newtheorem{theorem}{Theorem}[section]
\newtheorem{lemma}[theorem]{Lemma}
\newtheorem{corollary}[theorem]{Corollary}
\newtheorem{proposition}[theorem]{Proposition}
\theoremstyle{definition}
\newtheorem{definition}[theorem]{Definition}
\newtheorem{remark}[theorem]{Remark}

\def\bdef{\begin{definition}}
\def\endef{\end{definition}}
\def\bthm{\begin{theorem}}
\def\ethm{\end{theorem}}
\def\blm{\begin{lemma}}
\def\elm{\end{lemma}}
\def\brm{\begin{remark}}
\def\erm{\end{remark}}
\def\bprop{\begin{proposition}}
\def\eprop{\end{proposition}}
\def\bcor{\begin{corollary}}
\def\ecor{\end{corollary}}
\def\beq{\begin{eqnarray}}
\def\eeq{\end{eqnarray}}
\def\beal{\begin{aligned}}
\def\enal{\end{aligned}}
\def\beaa{\begin{eqnarray*}}  
\def\eeaa{\end{eqnarray*}}

\def\Om{\Omega}
\def\al{\alpha}

\def\phi{\varphi}

\def\~{\tilde}

\def\beq {\begin{equation}}
\def\eeq {\end{equation}}
\def\bdef{\begin{definition}}
\def\endef{\end{definition}}
\def\blm{\begin{lemma}}
\def\elm{\end{lemma}}
\def\beal{\begin{aligned}}
\def\enal{\end{aligned}}

\title[On the finite dimensionality of integrable deformation]{On the finite dimensionality of 
integrable deformations of strictly convex integrable billiard tables}
\begin{document}
\author{Guan Huang}
\address{Yau Mathematical Sciences Center, Tsinghua University, Beijing, China}
\email{ghuang@math.tsinghua.edu.cn}

\author{Vadim Kaloshin}
\address{Department of Mathematics, University of Maryland, College Park, MD, USA}
\email{vadim.kaloshin@gmail.com}

\maketitle
\begin{abstract}In this paper, we show that any smooth one-parameter deformations of a strictly convex integrable billiard table $\Omega_0$ preserving  the integrability near the boundary have to be tangent to a finite dimensional  space passing through~$\Omega_0$.
\end{abstract}

\medskip 
\begin{center}
{\it \quad \quad \quad 
Dedicated to Yulij Ilyashenko on his 75th birthday }
\end{center}
\medskip 

\section{Introduction}
A billiard system (\cite{BIR27}) consists by the inertial motions of a point mass inside a fixed domain and the elastic reflections at the boundary. Let $\Omega$ be a strictly convex domain in $\mathbb{R}^2$ with $C^r$ boundary $\partial \Omega$, with $r\geqslant3$. The phase space $M$ of the induced billiard system is a (topological) cylinder formed by the pair $(x,v)$, with $x$ being a foot point on $\partial\Omega$ and $v$ being an inward unit vector. The billiard ball map $f: \;M\to M$ takes $(x,v)$ to $(x',v')$, where $x'$ is the position on the boundary $\partial\Omega$, where the trajectory of  the point mass starting at $x$ with velocity $v$ first hits, and $v'$ is the reflected velocity, according to the standard reflection law of light: the angle of incidence is equal to the angle of reflection. For a systematic introduction to the billiard dynamics, see e.g. \cite{GUT03, TAB95,TAB05}.

A smooth convex   curve $\Gamma\subset\Omega$ is called a caustic, if whenever a trajectory is tangent to it, then they remain tangent after each reflection. Notice that each convex caustic $\Gamma$ corresponds to an invariant curve of the associated billiard map $f$ and, hence, has a well-defined rotation number.  
If the union of all the caustics form a set with non-empty interior, then we call the billiard table integrable. The famous Birkhoff conjecture (\cite{POR50}) claims that every integrable billiard table has a circle or an ellipse as its boundary. Though much attention it has attracted, this conjecture remains open, and only a few partial progresses were obtained. As far as our understanding of integrable billiards is concerned, the most important related results are 1) a theorem (\cite{BIA93}) by  Bialy which asserts that if the phase space of a billiard map is almost everywhere foliated by non-null homotopic invariant curves, then the corresponding billiard table is a disk; 2) a result (\cite{INN02}) by Innami, in which he showed that if a strictly convex billiard table admits a sequence of smooth convex caustics with rotation numbers converge to 1/2, then its boundary has to be an ellipse; 3) a result (\cite{DR96}) by Delshams and Ram\'irez-Ros in which they study entire perturbations of elliptic billiards and prove that any nontrivial symmetric perturbations of the elliptic billiard is not integrable (see also \cite{DR13,RAM06}); 4) and the more recent works (\cite{ADK16, HKS18,KS18}) by Kaloshin et al., justifying a perturbative version of the Birkhoff conjecture for  billiard tables with  boundary close to ellipses, assuming integrability near the boundary. 

In this work, we study deformation of a strictly convex integrable billiard table, which may not be closed to an ellipse.
Let us introduce some  notions of this paper.
\begin{definition}
(i) We say that $\Gamma\subset\Omega$ is an integrable rational caustic for the billiard system in $\Omega$ if the corresponding (non-contractible) invariant curve consists of periodic points; in particular, the corresponding rotation number is rational.

(ii) Let $q_0\geqslant2$. If the billiard system induced by $\Omega$ admits integrable rational caustics of rotation number $p/q$ for all $0<p/q<1/q_0$, we say that $\Omega$ is $q_0$-rationally integrable.
\end{definition}
\begin{remark} Let $\mathcal{C}_\Omega$ denote the union of all smooth convex caustics of the billiard in $\Omega$; if the interior of $\mathcal{C}_\Omega$ contains caustics of rotation numbers $p/q$ for all $0<p/q<1/q_0$, then $\Omega$ is $q_0$-rationally integrable. See \cite[Lemma1]{ADK16}.
\end{remark}
The main result of this work is  the following:
\begin{theorem} \label{thm:conclusion}
Let $\Om_0$ be a strictly convex $C^r$-smooth ($r\geqslant8$) domain that  is $q_*$-rationally integrable. 
Then there is $q_0=q_0(\Om)\ge q_*,\ d=2q_0+1,$ and, in the space of strictly convex $C^r$-smooth domains,  
a $d$-dimensional space $\mathcal T(\Omega_0)$ passing through $\Om_0$ such that 
any smooth deformation $\{\Om_t\}_t$ with $\Om_t$ being $q_0$-rationally integrable is tangent 
to $\mathcal T$. 
\end{theorem} 
\begin{remark}
For the proof, we only need the preservation of integrable  caustics with rotation numbers $1/q$, $q=q_0,\;q_0+1,\dots$.  Our approach here is inspired by those in~\cite{DKW17}. Namely, we  first derive the necessary  annihilation conditions (see Proposition \ref{vanish}) for the infinitesimal deformation function (see \eqref{deformation-func}) of the integrable deformation,  then using these constraints we construct an operator (see \eqref{operator-def}) which is invertible for suitable $q_0$ (see Theorem \ref{thm:main}), and finally from the invertibility of the operator, we conclude that the infinitesimal deformation function must belong to certain linear space of finite dimension. 
\end{remark}
\begin{remark}The most illuminating example of this theorem is the integrable deformations of  a domain with an ellipse as its boundary.  The  domain enclosed by   an ellipse is  2-rationally integrable.  Due to \cite{KS18}, 
any smooth one parameter family of deformations of this domain, preserving the $2$-integrability, are consisted of a family of domains with ellipses as their boundary, belonging to a $5$-dimensional space. 
\end{remark}
 Theorem \ref{thm:conclusion} can be viewed 
as a finite-dimensional reduction for integrable deformations. More explicit bounds on the dimension $d$ is subjected to future development.

\section{Necessary  conditions for the preservation of caustics}
From now on, we restrict ourself to strictly convex $C^r$-domains with $r\geqslant8$. Consider a one-parameter smooth deformation $\Omega_\tau$, $\tau\in[-1,1]$, of the strictly convex domain $\Omega_0$, preserving the existence of an integrable caustic with rotation number~$\frac{1}{q}$,~$q>2$. Let $\Gamma(\tau,\xi)$ be a parametrization of $\Omega_{\tau}$. As in \cite{DKW17}, we define the infinitesimal deformation function
\begin{equation}\label{deformation-func}n_{\Gamma}(\tau,\xi)=\langle \partial_{\tau}\Gamma(\tau,\xi), N_{\Gamma}(\tau,\xi)\rangle,
\end{equation}
where $\langle\cdot,\cdot\rangle$ is the usual scalar product in $\mathbb{R}^2$ and $N_{\Gamma}(\tau,\xi)$ is the outgoing unit normal vector to $\partial\Omega_{\tau}$ at the point $\Gamma(\tau,\xi)$. Note that $n_\Gamma$ is continuous in $\tau$, and $n_\Gamma(\tau,\cdot)\in C^r(\mathbb{T}^1,\mathbb{R})$ for each $\tau\in[-1,1]$.

Let $S_q(\tau,\xi)=(\xi_{\tau,k}^q,\varphi_{\tau, k}^q)_{k=0,\dots,q-1}$ be a periodic orbit of the billiard map induced by $\Omega_\tau$, where the starting point is $\Gamma(\tau,\xi)$, that is $\xi_{\tau,0}=\xi$, and $\varphi_{\tau,k}^q$ is the angle between the trajectory and the tangent line of $\Omega_\tau$ at $\xi_{\tau,k}$.  For any $C^1$-smooth function $\nu:\mathbb{T}^1\to\mathbb{R}$, we define
$$L^q_{\Gamma_\tau}(\nu)(\xi)=\sum_{k=0}^{q-1}\nu(\xi_{\tau,k}^q)\sin\varphi_{\tau,k}^q.$$
\begin{proposition}\label{vanish} The function $L^q_{\Gamma_\tau}\Big(n_{\Gamma}(\tau,\cdot)\Big)(\xi)$ is a constant with respect to~$\xi$. In particular, 
$$\frac{d}{d\xi}L^q_{\Gamma_\tau}\Big(n_{\Gamma}(\tau,\cdot)\Big)(\xi)=0.$$
\end{proposition}
\begin{proof} Let $\mathcal{L}^q(\tau, \xi)$ be the perimeter of the periodic orbit $S_q(\tau,\xi)$. From  \cite[ Proposition 4.6]{DKW17} we have that
$$\partial_{\tau}\mathcal{L}^q(\tau,\xi)=L_{\Gamma_\tau}^q\Big(n_\Gamma(\tau,\cdot)\Big)(\xi).$$
Since $\Omega_\tau$ is a one-parameter family preserving the existence of an integrable caustic with rotation number $\frac{1}{q}$, we have that $\mathcal{L}^q(\tau,\xi)$ is a function independent of $\xi$. So does $L_{\Gamma_\tau}^q\Big(n_\Gamma(\tau,\cdot)\Big)(\xi)$. 
Moreover, 
$\frac{d}{d\xi}L^q_{\Gamma_\tau}\Big(n_{\Gamma}(\tau,\cdot)\Big)(\xi)=0.$\end{proof}
The above statement is true for any parametrization of the boundary $\partial\Omega_\tau$. Now we fix $\tau=0$ and choose the {\it Lazutkin parametrization $x$} (\cite{LAZ73}), which is particularly convenient when dealing with nearly glancing orbits (e.g. periodic orbits with rotation number $\frac{1}{q}$, when $q$ is large). 
 Let $s$ be the length parameter of the boundary $\partial \Omega_0$ and $\rho(s)$ is the radii of curvature of $\partial\Omega_0$ at $s$. Note that $\rho$ is $C^{r-2}$, since $\Omega_0$ is $C^r$. Then
the  Lazutkin parametrization  of the boundary $\partial\Omega_0$ is given as follows:
\[x(s)=C_{\Omega_0}\int_0^s\rho^{-2/3}(\tau)d\tau,\quad \text{ and }\quad C_{\Omega_0}^{-1}=\int_0^{|\partial\Omega_0|}\rho^{-2/3}(s)ds.\]
We introduce the {\it Lazutkin density}: 
$$\mu(x)=\frac{1}{2C_{\Omega_0}\rho(x)^{1/3}},$$
when we denote $\rho(x)=\rho(s(x))$, the radius of curvature in the Lazutkin parametrization. The following statement was obtained in \cite[Appendix A and B]{DKW17}.
\begin{lemma}\label{lazutkin}
Assume $r\geqslant 8$. There exist constant $C=C(\Omega_0)$ and $1$-periodic functions $\alpha(x)$ and $\beta(x)$ such that for each $q\geqslant2$, there exist $1$-periodic functions $\gamma_{a,\frac{k}{q}}(x)$, $\gamma_{b,\frac{k}{q}}(x)$, $\gamma_{c,\frac{k}{q}}$, $k=1,\dots q$ such that 
$$\|\alpha\|_{C^{r-4}},\|\beta\|_{C^{r-4}},\|\gamma_{a,\bullet}\|_{C^{r-6}},\;\|\gamma_{b,\bullet}\|_{C^{r-6}},\;\|\gamma_{c,\bullet}\|_{C^{r-6}}\leqslant C,$$
and for any periodic orbit $x=x_q^0,\dots, x_q^{q-1}(x)$ with rotation number $\frac{1}{q}$, we have
$$x_q^{k}(x)=x+\frac{k}{q}+\frac{\alpha(x+\frac{k}{q})}{q^2}+\frac{1}{q^4}\gamma_{a,\frac{k}{q}}(x).$$
Moreover, if $\varphi_q^k(x)$ denotes the angle of reflection of the trajectory at the $k$-th collision, we have
$$\varphi_{q}^k(x)=\frac{\mu(x_q^k)}{q}\Big(1+\frac{\beta(x+\frac{k}{q})}{q^2}+\frac{1}{q^4}\gamma_{b,\frac k q}(x)\Big).$$
and
$$\frac{\sin\varphi_q^k(x)}{\mu(x_q^k(x))}=\frac{1}{q}\Big[1+\frac{\beta(x+\frac{k}{q})}{q^2}+S_q(x+\frac{k}{q})+\frac{1}{q^4}\gamma_{c,k/q}(x)\Big],$$
where 
\begin{align}\label{e_S-function}
S_q(x)=
\frac{\sin \left( \mu(x)/q\right)}{\mu(x)/q}-1.
    \end{align}
\end{lemma}

\begin{remark}\label{appendix-a-b}
(i) Notice that 
$S_q(x)
=\frac{\left( \mu(x)/q\right)^2}{6}(-1+O(q^{-2})).$

(ii) The functions $\alpha(x)$ and $\beta(x)$ satisfy 
$$\alpha'(x)=\beta(x)+\frac{-\rho^{1/3}(x)\rho''(x)}{36 C_{\Omega_0}^2}+\frac{\rho^{-2/3}(x)(\rho'(x))^2}{54C_{\Omega_0}^2}.$$
See Lemma \ref{lem:alpha-beta}.
\end{remark}

Now denoting $\nu_0(x)=n_{\Gamma}(0,x)\mu(x)$,
we  use Lemma \ref{lazutkin} to write the quantity $$\frac{d}{dx}L_{\Gamma_0}^q\big(n_\Gamma(0,\cdot)\big)(x)$$ more explicitly,
%
\[\begin{split}&\frac{d}{dx}L_{\Gamma_0}^q\big(n_\Gamma(0,\cdot)\big)(x)=\frac{d}{dx}\sum_{k=0}^{q-1}n_{\Gamma}(0,x_q^k(x))\mu(x_q^k(x))\frac{\sin\varphi_q^k(x)}{\mu(x_q^k(x))}\\
&=\sum_{k=0}^{q-1}\nu_0'
\left(x+\frac{k}{q}+\frac{\alpha(x+\frac{k}{q})}{q^2}+\frac{1}{q^4}\gamma_{a,k/q}(x)\right)
\times\Big(1+\frac{\alpha'(x+\frac{k}{q})}{q^2}+\frac{1}{q^4}\gamma_{a,k/q}'(x)\Big)\\
&\quad\quad \quad \times \frac{1}{q}\Big[1+\frac{\beta(x+\frac{k}{q})}{q^2}+S_q(x+\frac{k}{q})+\frac{1}{q^4}\gamma_{c,k/q}(x)\Big]\\
&
\quad+\sum_{k=0}^{q-1}\nu_0\left(x+\frac{k}{q}+\frac{\alpha(x+\frac{k}{q})}{q^2}+\frac{1}{q^4}\gamma_{a,k/q}(x)\right)\\
&\quad\quad\quad
\times\frac{1}{q}\Big[\frac{\beta'(x+\frac{k}{q})}{q^2}+S_q'(x+\frac{k}{q})+\frac{1}{q^4}\gamma'_{c,k/q}(x)\Big].
\end{split}\]
For any $C^1$-function $f:\mathbb{T}^1\to\mathbb{R}$, we 
consider two linear operators:
\begin{equation}\label{lq1-def}\begin{split}\mathbb{L}_q^1(f)(x)=&\frac{1}{2\pi q^2}\sum_{k=0}^{q-1}f'\Big(x+\frac{k}{q}+\frac{\alpha(x+\frac{k}q)}{q^2}+\frac{1}{q^4}\gamma_{a,\frac{k}{q}}(x)\Big)\\
&\quad\times\Big(1+\frac{\alpha'(x+\frac{k}{q})}{q^2}+\frac{1}{q^4}\gamma_{a,k/q}'(x)\Big)\\
&\quad\times \Big(1+\frac{\beta(x+\frac{k}{q})}{q^2}+S_q(x+\frac{k}{q})+\frac{1}{q^4}\gamma_{c,k/q}(x)\Big),
\end{split}\end{equation}
and
\begin{equation}\label{lq2-def}\begin{split}\mathbb{L}_q^2(f)(x)=&\frac{1}{2\pi q^2}\sum_{k=0}^{q-1}f\Big(x+\frac{k}{q}+\frac{\alpha(x+\frac{k}{q})}{q^2}+\frac{1}{q^4}\gamma_{a,k/q}(x)\Big)\\
&\quad\quad
\times\Big(\frac{\beta'(x+\frac{k}{q})}{q^2}+S_q'(x+\frac{k}{q})+\frac{1}{q^4}\gamma'_{c,k/q}(x)\Big)
\end{split}\end{equation}

By definition, we have
\begin{equation}\label{operator-eq1}\frac{d}{dx}L_{\Gamma_0}^q\big(n_\Gamma(0,\cdot)\big)(x)=2\pi q\Big[\mathbb{L}_q^1(\nu_0)(x)+\mathbb{L}_q^2(\nu_0)(x)\Big].\end{equation}
For any $f\in C^1(\mathbb{T})$, let 
$$f(x)=a_0+\sum_{k=1}^{+\infty}a_k\cos2\pi kx+b_k\sin2\pi kx$$
be its Fourier series. 
For $2<\gamma<3$, define  a subspace $X^\gamma\subset C^0(\mathbb{T})$ as 
$$X^\gamma =\{f\in C^0(\mathbb{T}):\;\|f\|_{X^\gamma}<+\infty\},$$
with
$$\|f\|_{X^\gamma}= |a_0|\wedge\Big(\sup_{j\geqslant1}(j^\gamma|a_j|\wedge |b_j|)\Big)$$
where $a\wedge b=\max\{a,b\}$. 
The space $(X^\gamma, \|\cdot\|_{X^\gamma})$ is a (separable) Banach space.
\begin{remark} We have $C^3(\mathbb{T})\subset X^\gamma\subset C^1(\mathbb{T})$, since $2<\gamma<3$. So the linear operators $\mathbb{L}_q^1$ and $\mathbb{L}_q^2$ are well defined on $X^\gamma$. 
\end{remark}
 Let us introduce another (separable) Banach space $h^\gamma\subset\mathbb{R}^\infty$,  
 $$h^\gamma:=\{c=(c_0, c_1,d_1,c_2,d_2,\dots)\in\mathbb{R}^{\infty}: \|c\|_{h^\gamma}
<+\infty\},$$
equipped with the norm
$$\|c\|_{h^\gamma}=|c_0|\wedge\Big(\sup_{j\geqslant1}(j^\gamma |c|_j\wedge|d_j|)\Big).$$
For some $\bar q\geqslant2$, we define the following linear map,
\begin{equation}\label{operator-def}\begin{split}&\mathcal{I}^{\bar q}: X^\gamma\to \mathbb{R}^\infty, \\
f\mapsto&\Big(a_0,\dots,a_{\bar q-1},b_{\bar q-1}, \\
\mathbb{L}_{\bar q}(f)(\frac{3}{4\bar q}),\;&\mathbb{L}_{\bar q}(f)(0), \dots, 
 \mathbb{L}_{q}(f)(\frac{3}{4q}),\;\mathbb{L}_{q}(f)(0),\dots\Big),\end{split}\end{equation}
where
$$\mathbb{L}_q(f)(x)=\mathbb{L}_q^1(f)(x)+\mathbb{L}_q^1(f)(x).$$
Let $\gamma_0$ be the number\footnote{$\gamma_0\approx 2.78831...$} such that 
$$\sum_{k=2}^{+\infty}\frac{1}{k^{\gamma_0-1}}=0.9.$$
Observe that  $\frac{8}{3}<\gamma_0<3$. The our main theorem is implied by the following statement:
\begin{theorem} \label{thm:main}
For any $\gamma$ such that $\gamma_0\leqslant\gamma<3$, there exists $q_0= q_0(\Omega_0,\gamma)$  such that for each $\bar q\geqslant q_0$, 
we have that the linear map
$$
\mathcal{I}^{\bar q}(f):X^\gamma \to h^\gamma
$$ 
is invertible. 
\end{theorem}
{\it Proof of Theorem \ref{thm:conclusion}.}  
 Fix $\gamma\geqslant\gamma_0$. Let $\Omega_\tau$ be one parameter $\bar q$-rationally integrable deformations of the domain $\Omega_0$, with $\bar{q}\geqslant q_0(\Omega_0,\gamma)$. Then by Theorem~\ref{thm:main} and Proposition \ref{vanish}  we have that 
$$\nu_0\in \mathcal{I}_d=[\mathcal{I}^{\bar{q}}]^{-1}\Big(h_{\bar q}^{\gamma}\Big)\subset C^1(\mathbb{T}),$$
where $h_{\bar q}^{\gamma}=\{c\in h^{\gamma}: c_{j}=d_j=0, \forall j\geqslant \bar q\}$ is a finite dimensional subspace of~$h^{\gamma}$. 
Since $\nu_0(\cdot)=\mu(\cdot)n_{\Gamma}(0,\cdot)$, and $\mu(x)$ is a nowhere vanishing function because of the fact that $\Omega_0$ is strictly convex, we have that the infinitesimal deformation function $n_\Gamma(0,\cdot)$ belongs to the set $\mathcal{I}=\mu^{-1}\mathcal{I}_d$, which is a finite dimensional linear subspace of $C^1(\mathbb{T})$.  Hence the assertion of Theorem \ref{thm:conclusion} follows. $\square$

\medskip

In the rest of this paper, we focus on the proof of Theorem \ref{thm:main}. In Section~\ref{lq}, we obtain estimates for the linear operators $\mathbb{L}_q^1$ and $\mathbb{L}_q^2$, applying to the base functions $\cos2\pi px$ and $\sin2\pi px$. Then in Section \ref{proofthm} we complete the proof of Theorem \ref{thm:main}.

\section{Estimates for the operators $\mathbb{L}_q^1$ and $\mathbb{L}^2_q$}\label{lq}
In this section we study the linear operator $\mathbb{L}_q^1$ and $\mathbb{L}_q^2$, which are defined in \eqref{lq1-def} and \eqref{lq2-def}.   

We first consider $\mathbb{L}_q^1(f)(x)$ with $f(\cdot)=\cos 2\pi p(\cdot) $. Then, 
\[\begin{split}&\mathbb{L}_q^1\big(\cos2\pi(\cdot)\big)(x)\\
&=\frac{1}{q^2}\sum_{k=0}^{q-1}
-
 p\sin\left(2\pi p(x+\frac{k}{q}+\frac{\alpha(x+\frac{k}{q})}{q^2})\right) \\
&\quad\quad\quad \times\Big[1+\frac{\alpha'(x+\frac{k}{q})+\beta(x+\frac{k}{q})}{q^2}+
S_q(x+\frac{k}{q})\Big]+O(\frac{p^2}{q^5})\\
&=\frac{1}{q^2}\sum_{k=0}^{q-1}- p
\sin 2\pi p\left (x+\frac{k}{q}+\frac{\alpha(x+\frac{k}{q})}{q^2}\right)\\
&\quad\quad\quad
\times\Big[1+\frac{\alpha'(x+\frac{k}{q})+\beta(x+\frac{k}{q})}{q^2}-
\frac{\mu^2(x+\frac{k}{q})}{6q^2}\Big]+O(\frac{p^2}{q^5})
\end{split}\]
where we have used $S_q(x)=\frac{\mu^2(x)}{6q^2}(-1+O(\frac{1}{q^2}))$. 
 Note that for $p>q^2$, then $O(\frac{p^2}{q^5})$ could be replaced by $O(\frac{p}{q})$.
  Let us  set 
\[\Lambda_1^q(\cos2\pi px)=
\frac{1}{q^2}\sum_{k=0}^{q-1}-p\sin 2\pi p\left(x+\frac{k}{q}+\frac{\alpha(x+\frac{k}{q})}{q^2}\right).
\]
and 
\[\Lambda_2^q(\cos2\pi px)=\frac{1}{q^4}\sum_{k=0}^{q-1}-p\sin2\pi p\Big(x+\frac{k}{q}+\frac{\alpha(x+\frac{k}{q})}{q^2}\Big)\Big[\alpha'(x+\frac{k}{q})+\beta(x+\frac{k}{q})-\frac{\mu^2(x+\frac{k}{q})}{6}\Big].\]
Clearly, 
$$\mathbb{L}_q^1\big(\cos2\pi p(\cdot)\big)(x)=\Lambda_1^q(\cos2\pi px)+\Lambda_2^q(\cos2\pi px)+O(\frac{p^2}{q^5}).$$
For $\Lambda_2^q$, the following simple inequality is enough for our purpose here, 
$$|\Lambda_2^q(\cos2\pi px)|\leqslant C_1\frac{p}{q^3},$$
where $C_1$ depends only on the $C^0$-norm of $\alpha'$, $\beta$ and $\mu$.

Now let us study $\Lambda_1^q(\cos2\pi px)$.  If $p=q$, we have that
\[|\Lambda_1^q(\cos2\pi qx)+\sin 2\pi qx|\leqslant C_2\frac{1}{q}.\]
where $C_2$ depends only on the $C^0$-norm of $\alpha$.

For $p\neq q$, let us consider the function
$$
F_q^p(x)=\sin 2\pi p(x+\frac{\alpha(x)}{q^2})$$
and its Fourier series
$$
F_q^p(x)=a_0^{p,q}+\sum_{k=1}^{+\infty}a_k^{p,q}\cos2\pi kx+b_k^{p,q}\sin2\pi kx.
$$
Then
$$\Lambda_1^q(\cos2\pi px)=\frac{-p}{q}a_0^{p,q}+\frac{-p}{q}\sum_{j=1}^{\infty}a_{jq}^{p,q}\cos2\pi jqx+b_{jq}^{p,q}\sin2\pi jqx.$$
Now we estimate each quantity. Let us begin with $a_0^{p,q}$,
\begin{equation}\label{a0pq}\begin{split}a_0^{p,q}&=\int_0^1\sin 2\pi p(x+\frac{\alpha(x)}{q^2})dx\\
&=\int_0^1\Big[\sin2\pi px\cos2\pi p\frac{\alpha(x)}{q^2}+\cos2\pi px\sin2\pi p\frac{\alpha(x)}{q^2}\Big]dx\end{split}\end{equation}
For the first term, using integration by part, we have
\[\begin{split}&\int_0^1\sin2\pi px\cos2\pi\frac{p\alpha(x)}{q^2}dx\\
&=\int_0^1-\frac{\cos2\pi px}{2\pi p}[\sin2\pi \frac{p\alpha(x)}{q^2}]\frac{2\pi p\alpha'(x)}{q^2}dx\\
&=\int_0^1\frac{\sin2\pi px}{2\pi p}\Big[(\cos2\pi \frac{p\alpha(x)}{q^2})\frac{2\pi p(\alpha'(x))^2}{q^4}+(\sin2\pi \frac{p\alpha(x)}{q^2})\frac{\alpha''(x)}{q^2}\Big]dx.
\end{split}\]
Similarly, for the second term in \eqref{a0pq}, we have 
\[\begin{split}&\int_0^1\cos2\pi px\sin2\pi \frac{p\alpha(x)}{q^2}dx\\
&=\int_0^1\frac{\cos2\pi px}{2\pi p}[(\sin2\pi\frac{p\alpha(x)}{q^2})\frac{2\pi p(\alpha'(x))^2}{q^4}-(\cos2\pi \frac{p\alpha(x)}{q^2})\frac{\alpha''(x)}{q^2}]dx.\end{split}\]
Therefore,
\[|a_0^{p,q}|\leqslant C_3(\frac{1}{q^4}+\frac{1}{pq^2}),\]
where the constant $C_3>0$ depends only on $C^2$-norm of $\alpha$.

For the quantity $a_{kq}^{p,q}$, 

\[\begin{split}a_{kq}^{p,q}&=\int_0^1\sin2\pi p(x+\frac{\alpha(x)}{q^2})\cos2\pi kq xdx\\
&=\int_0^{1}\Big[\sin2\pi px\cos2\pi p\frac{\alpha(x)}{q^2}+\cos2\pi px\sin2\pi p\frac{\alpha(x)}{q^2}\Big]\cos2\pi kqdx\\
&=\int_0^1(\cos2\pi p\frac{\alpha(x)}{q^2})\Big(\frac{\sin2\pi(kq+p)x-\sin2\pi (kq-p)x}{2}\Big)dx\\
&\quad \quad +(\sin2\pi p\frac{\alpha(x)}{q^2})\Big(\frac{\cos2\pi (p+kq)x+\cos2\pi (kq-p)x}{2}\Big)dx
\end{split}\]
Hence there exists $C_4>0$ depending only on $C^2$-norm of $\alpha$ such that 
if $p\neq nq$ for all $n\in \mathbb{N}$, then 
$$|a_{kq}^{p,q}|\leqslant C_4\frac{1}{|kq-p|^2}(\frac{p}{q^2}+\frac{p^2}{q^4}),$$
and if $p=kq$ for some $k\in\mathbb{N}$, then
$$|a_{kq}^{p,q}-\frac{1}{2}\int_0^1\sin2\pi p\frac{\alpha(x)}{q^2}dx|\leqslant C_4\frac{1}{|kq+p|^2}\Big|\frac{p}{q^2}+\frac{p^2}{q^4}\Big|.$$
So if $p=nq$ for some $n>1$, then we have
\[\begin{split}&\sum_{k=1}^{\infty}a_{kq}^{p,q}\cos2\pi kqx\\
&=\Big(\frac{1}{2}\int_0^1\sin2\pi p\frac{\alpha(x)}{q^2}dx\Big)\cos2\pi nqx+O((\frac{p}{q^2}+\frac{p^2}{q^4})\sum_{k=1}^{\infty}\frac{1}{k^2q^2}).\\
&=\Big(\frac{1}{2}\int_0^1\sin2\pi p\frac{\alpha(x)}{q^2}dx\Big)\cos2\pi nqx+O(\frac{p}{q^4}+\frac{p^2}{q^6}).\end{split}\]
If $p=nq+r$ for some $n\geqslant0$ and $1\leqslant r\leqslant q-1$,
\[\begin{split}&\sum_{k=1}^{\infty}a_{kq}^{p,q}\cos2\pi kqx\\
&=O\Big((\frac{p}{q^2}+\frac{p^2}{q^4})\sum_{k=1}^{\infty}\frac{1}{q^2(k+\frac{r}{q})^2}+(\frac{p}{q^2}+\frac{p^2}{q^4})(\frac{1}{r^2}+\frac{1}{(q-r)^2})\Big)\\
&=O(\frac{p}{q^4}+\frac{p^2}{q^6})+
O\left((\frac{p}{q^2}+\frac{p^2}{q^4})(\frac{1}{r^2}+\frac{1}{(q-r)^2})\right).
\end{split}\]
\[\begin{split}
\end{split}\]
Similarly, we have 
\[\begin{split}b_{kq}^{p,q}&=\int_0^1\sin2\pi p(x+\frac{\alpha(x)}{q^2})\sin2\pi kqxdx\\
&=\begin{cases}O(|\frac{p}{q^2}+\frac{p^2}{q^4}|\frac{1}{|kq-p|^2}),&p\not\in q\mathbb{N},\\
-\int_0^1\cos2\pi p\frac{\alpha(x)}{q^2}dx+O(|\frac{p}{q^2}+\frac{p^2}{q^4}|\frac{1}{|kq+p|^2}),&p\in q\mathbb{N}.\end{cases}
\end{split}\]
Then, if $p=nq$ for some $n>1$, we have 
\[\begin{split}&\sum_{k=1}^{\infty}b_{kq}^{p,q}\cos2\pi kqx\\
&=\Big(-\frac{1}{2}\int_0^1\cos 2\pi p\frac{\alpha(x)}{q^2}dx\Big)\cos2\pi nqx+O((\frac{p}{q^2}+\frac{p^2}{q^4})\sum_{k=1}^{\infty}\frac{1}{k^2q^2}).\\
&=\Big(-\frac{1}{2}\int_0^1\cos2\pi p\frac{\alpha(x)}{q^2}dx\Big)\cos2\pi nqx+O(\frac{p}{q^4}+\frac{p^2}{q^6}),\end{split}\]
and if $p=nq+r$ for some $n\geqslant0$ and $1\leqslant r\leqslant q-1$, we have
\[\begin{split}&\sum_{k=1}^{\infty}b_{kq}^{p,q}\cos2\pi kqx\\
&=O\Big((\frac{p}{q^2}+\frac{p^2}{q^4})\sum_{k=1}^{\infty}\frac{1}{q^2(k+\frac{r}{q})^2}+(\frac{p}{q^2}+\frac{p^2}{q^4})(\frac{1}{r^2}+\frac{1}{(q-r)^2})\Big)\\
&=O(\frac{p}{q^4}+\frac{p^2}{q^6})+
O\left((\frac{p}{q^2}+\frac{p^2}{q^4})(\frac{1}{r^2}+\frac{1}{(q-r)^2})\right).
\end{split}\]
Therefore, for $\mathbb{L}_q^1(\cos2\pi px)$, we have that 
\begin{lemma}\label{lemma-lq1-c}
\begin{enumerate}
\item If $p=0$, $\mathbb{L}_q^1(1)=0$.
\item If $p=q$, we have 
\[\mathbb{L}_q^1\big(\cos2\pi q(\cdot)\big)(x)=-\sin 2\pi qx+O(\frac{1}{q}).\]

\item If $p=qk$ for some $1<k\in\mathbb{N}$, 
\[\begin{split}
\mathbb{L}_q^1\big(\cos2\pi kq(\cdot)\big)(x)=&\frac{- kq}{q}\Big[\Big(\frac{1}{2}\int_0^1\sin2\pi kq\frac{\alpha(x)}{q^2}dx\Big)\cos2\pi kqx\\
&\quad \quad +\Big(-\frac{1}{2}\int_0^1\cos2\pi kq\frac{\alpha(x)}{q^2}dx\Big)\sin2\pi kqx\Big]\\
&\quad \quad +O(\frac{1}{q^3}+\frac{p^2}{q^5}+\frac{p^3}{q^7}+\frac{p}{q^3}).
\end{split}\]
  \item If $p=nq+r$ for some $n\in\mathbb{N}$ and $1\leqslant r\leqslant n-1$,
\[\mathbb{L}_q^1\big(\cos2\pi p(\cdot)\big)(x)=O(\frac{1}{q^3}+\frac{p^2}{q^5}+\frac{p^3}{q^7})+O((\frac{p^2}{q^3}+\frac{p^3}{q^7})(\frac{1}{r^2}+\frac{1}{(q-r)^2}))\]
\end{enumerate}
\end{lemma}
Now consider $f=\sin2\pi px$. Then
\[\begin{split}&\mathbb{L}_q^1\big(\sin2\pi p(\cdot)\big)(x)\\
&=\frac{1}{q^2}\sum_{k=0}^{q-1} p\cos2\pi p(x+\frac{k}{q}+\frac{\alpha(x+\frac{k}{q})}{q^2})\\
&\quad\quad \times\Big[1+\frac{\alpha'(x+\frac{k}{q})+\beta(x+\frac{k}{q})}{q^2}-\frac{\mu^2(x+\frac{k}{q})}{6q^2}\Big]+O(\frac{p^2}{q^5})
\end{split}\] 
Then similar to $\mathbb{L}_q^1\big(\cos2\pi p(\cdot)\big)$, we have that
\begin{lemma}\label{lemma-lq1-s}
\begin{enumerate}
\item If $p=q$, then 
\[\mathbb{L}_q^1\big(\sin2\pi q(\cdot)\big)(x)= \cos2\pi qx+O(\frac{1}{q}).\]
\item If $p=kq$, then
\[\begin{split}\mathbb{L}_q^1\big(\sin2\pi kq(\cdot)\big)(x)=&\frac{kq}{q}\Big[\Big(\frac{1}{2}\int_0^1\cos2\pi kq\frac{\alpha(x)}{q^2}dx\Big)\cos2\pi kqx\\
&\quad +\Big(\int_0^1\frac{1}{2}\sin2\pi \frac{kq\alpha(x)}{q^2}dx\Big)\sin2\pi kqx\Big]\\
&\quad+O(\frac{1}{q^3}+\frac{p^2}{q^5}+\frac{p^3}{q^7}+\frac{p}{q^3}).
\end{split}\]
 \item If $p=kq+r$, then
\[\mathbb{L}_q^1\big(\sin2\pi p(\cdot)\big)(x)=O(\frac{1}{q^3}+\frac{p^2}{q^5}+\frac{p^3}{q^7})+O\Big((\frac{p^2}{q^3}+\frac{p^3}{q^7})(\frac{1}{r^2}+\frac{1}{(q-r)^2})\Big).\] 
\end{enumerate}
\end{lemma}

We end this section with a straightforward estimate for the operator $\mathbb{L}_q^2$.
\begin{lemma}\label{lemma-lq2}
There exists  a constant $C_5>0$ which depends only on the $C^1$-norm of $\beta$, $\mu$ and $\gamma_{c,k.q}$, such that for $u=\cos 2\pi px$ or $u=\sin 2\pi px$, 
$$|\mathbb{L}_q^2(u)(x)|\leqslant C_5/q^3.$$
\end{lemma}

\section{proof of theorem \ref{thm:main}}\label{proofthm}
In this section we prove Theorem \ref{thm:main}, using the results in Section \ref{lq}.  Now we fix $\gamma\in(\gamma_0,3)$. 

The operator $\mathcal{I}^{\bar q}$ defined in \eqref{operator-def} can be identified as a pair formed by    two infinite matrices 
\[\Big(\Big[\begin{array}{cc}l_{qj}^{1,c} & l_{qj}^{1,s} \\ l_{qj}^{2,c} & l_{qj}^{2,s}\end{array}\Big]\Big)_{qj}\quad\text{and}\quad ((g_k^1,g_k^2))_k\]
%
where for $q<\bar q$, $$g_0=1,\; g_q^1=g_q^2=l_{qj}^{1,s}=l_{qj}^{2,c}=0\;\text{ and }\;l_{qj}^{1,c}=l_{qj}^{2,s}=\delta_{q,j},$$
and for $q\geqslant \bar q$, $g_q^1=\mathbb{L}_q(1)(\frac{3}{4q})$, $g^2_q=\mathbb{L}(1)(0)$,
$$\begin{cases}l_{qj}^{1,c}=\mathbb{L}_q(\cos2\pi jx)(\frac{3}{4q}),\quad l_{qj}^{1,s}=\mathbb{L}_q(\sin2\pi jx)(\frac{3}{4q}),&\\
 l_{qj}^{2,c}=\mathbb{L}_q(\cos2\pi jx)(0),\quad l_{qj}^{2,s}=\mathbb{L}_q(\sin2\pi jx)(0).\end{cases}$$
 For \begin{equation}\label{f-fourier}f=a_0+\sum_{k=1}^{+\infty}a_k\cos2\pi kx+b_k\sin2\pi kx,\end{equation}
 we have that 
 \[\mathcal{I}^{\bar q}(f)=(c_0,c_1,d_1,\dots)\in\mathbb{R}^{\infty},\]
 with 
 $$c_0=a_0,\; c_k=a_k,\;d_k=b_k, \quad k=1,\dots,\bar q-1,$$
 and
 $$c_q=g_q^1a_0+\sum_{j=1}^\infty a_kl_{qj}^{1,c}+b_kl_{qj}^{1,s},\quad d_q=g_q^2a_0+\sum_{j=1}^\infty a_{k}l_{qj}^{2,c}+b_k l_{qj}^{2,s},\quad q=\bar q,\dots.$$
 Then we have the  norm of $\mathcal{I}^{\bar q}$ as a linear operator from $X^\gamma$ to $h^\gamma$ is 
 \[\|\mathcal{I}^{\bar q}\|_{\gamma}=\sup_{q}q^\gamma\big[\sum_{j\geqslant1}j^{-\gamma}\max\{|l_{qj}^{1,c}|+|l_{qj}^{1,s}|,|l_{qj}^{2,c}|+|l_{qj}^{2,s}|\}+\max\{|g_q^1|,|g_q^2|\}].\]
 To show that $\mathcal{I}^{\bar q}$ is an invertible operator, we consider the operator
 $\mathcal{I}^{\bar q}-\mathbb{I}$,
 where $$\mathbb{I}(f)=(a_0,a_1,b_2,\dots, a_k,b_k,\dots),$$ if $f$ is given as in \eqref{f-fourier}.
 Clearly, if we show for $\bar q$ large enough, 
 $$\|\mathcal{I}^{\bar q}-\mathbb{I}\|<1,$$
 then $\mathcal{I}^{\bar q}$ is an invertible linear operator from $X^\gamma$ to $h^\gamma$. 
 
 From Lemmas \ref{lemma-lq1-c}, \ref{lemma-lq1-s} and~\ref{lemma-lq2}, we know that for $q\geqslant \bar q$, 
 $$g_q^1,g_q^2=O(\frac{1}{q^3}),$$
 and there exists $C>0$ such that
  if $j=q$, 
 \[|l_{qq}^{1,c}-1|\leqslant \frac{C}{q},\;|l_{qq}^{1,s}|\leqslant \frac{C}{q},\; |l_{qq}^{2,c}|=\frac{C}{q},\; |l_{qq}^{2,s}-1|\leqslant \frac{C}{q},\]
 if $j=kq$, $k=2,3,\dots$, 
 
 \[|l_{qj}^{1,c}|+|l_{qj}^{1,s}|\leqslant
 \]
 \[
 |\frac{j}{2q}\Big(\Big|\int_0^1\cos\frac{2\pi k\alpha(x)}{q}dx\Big|+\Big|\int_0^{1}\sin\frac{2\pi k\alpha(x)}{q}dx\Big|\Big)+C(\frac{j+1}{q^3}+\frac{j^2}{q^5}+\frac{j^3}{q^7}),\]
 \[|l_{qj}^{2,c}|+|l_{qj}^{2,s}|\leqslant
 \]
 \[
 |\frac{j}{2q}\Big(\Big|\int_0^1\cos\frac{2\pi k\alpha(x)}{q}dx\Big|+\Big|\int_0^{1}\sin\frac{2\pi k\alpha(x)}{q}dx\Big|\Big)+C(\frac{j+1}{q^3}+\frac{j^2}{q^5}+\frac{j^3}{q^7}),\]
 and if $j=kq+r$ for some $r, k\in\mathbb{N}$ and $0<r<q$, 
 $$|l_{qj}^{1,c}|+|l_{qj}^{1,s}|, |l_{qj}^{2,c}|+|l_{qj}^{2,s}|\leqslant C\Big(\frac{1}{q^3}+\frac{j^2}{q^5}+\frac{j^3}{q^7}+(\frac{j^2}{q^3}+\frac{j^3}{q^7})(\frac{1}{r^2}+\frac{1}{(q-r)^2})\Big).$$
 Now, letting $B\geqslant1$ be a constant to be determined later and using for $j\geqslant q^{B+1}$ that the simple estimate $$|\mathbb{L}_q(u)(x)|\leqslant C\frac{j}{q},\quad u=\cos2\pi jx,\; \text{or} \;u=\sin2\pi jx,$$ we have that 
 \[\begin{split}&\|\mathcal{I}^{\bar q}-\mathbb{I}\|_\gamma\\
 &\leqslant\sup_{q\geqslant \bar q}q^\gamma\Big[\sum_{k=2}^{+\infty}(kq)^{-\gamma}\frac{kq}{q}\frac{1}{2}\Big(\Big|\int_0^1\cos\frac{2\pi k\alpha(x)}{q}dx\Big|+\Big|\int_0^{1}\sin\frac{2\pi k\alpha(x)}{q}dx\Big|\Big)\\
 &\quad\quad+C\Big(\frac{1}{q^3}+\frac{q^{-\gamma}}{q}+\sum_{k=2}^{q^B}(kq)^{-\gamma}(\frac{kq+1}{q^3}+\frac{(kq)^2}{q^5}+\frac{(kq)^3}{q^7})\Big)\\
 &\quad\quad+C\sum_{r=1}^{q-1}\sum_{k=0}^{q^B}\ \ (kq+r)^{-\gamma}\times 
 \\
 &\quad\quad\quad\quad\Big(\frac{1}{q^3}+\frac{(kq+r)^2}{q^5}+\frac{(kq+r)^3}{q^7}+\frac{(kq+r)^2}{q^3}(\frac{1}{r^2}+\frac{1}{(q-r)^2})\Big)\\
 &\quad\quad+C\sum_{j=q^{B+1}}^{\infty}j^{-\gamma} \frac{j}{q}\Big].\end{split}\]
We simplified the right hand side of the above inequality and get
\begin{equation}\begin{split}\|\mathcal{I}^{\bar q}-\mathbb{I}\|_\gamma\leqslant&\sum_{k=2}^\infty\frac{1}{k^{\gamma-1}}+ \sup_{q\geqslant \bar q}C\sum_{j=q^{B+1}}^\infty q^{\gamma-1} j^{-\gamma+1}\\
&+\sup_{q\geqslant \bar q}C\Big(\frac{1}{q}+q^{\gamma-3}+q^\gamma\sum_{j=1}^{q^{B+1}}j^{-\gamma}(\frac{j}{q^3}+\frac{j^2}{q^5}+\frac{j^3}{q^7})\\
&\quad\quad+q^\gamma\sum_{r=1}^{q-1}\sum_{k=0}^{q^B}\frac{(kq+r)^{2-\gamma}}{q^3}(\frac{1}{r^2}+\frac{1}{(q-r)^2})\Big)\end{split}\end{equation}

Now we estimate each term on the right hand side.

For the first term, since $\gamma\geqslant\gamma_0$, 
we have that
$$\sum_{k=2}^\infty\frac{1}{k^{\gamma-1}}\leqslant0.9.$$
For second term, we have
\begin{equation}\label{term2}\sum_{j=q^{B+1}}^\infty q^{\gamma-1}j^{-\gamma+1}\leqslant C'(\gamma)q^{(B+1)(2-\gamma)+\gamma-1}.\end{equation}
For the third term,
\begin{equation}\label{term3}
\begin{split}
q^{\gamma}\sum_{j=1}^{q^{B+1}}j^{-\gamma}(\frac{j}{q^3}+\frac{j^2}{q^5}+\frac{j^3}{q^7})\leqslant 
\quad\quad\quad\quad\quad\quad
\\
C(\gamma)\Big(q^{\gamma-3}+q^{\gamma-5}q^{(B+1)(3-\gamma)}+q^{\gamma-7}q^{(B+1)(4-\gamma)}\Big).
\end{split}\end{equation}
For the fourth term,
\begin{equation}\label{term4}\begin{split}&q^\gamma\sum_{r=1}^{q-1}\sum_{k=0}^{q^B}\frac{(kq+r)^{2-\gamma}}{q^3}(\frac{1}{r^2}+\frac{1}{(q-r)^2})\\
&\leqslant C(\gamma)q^{B(3-\gamma)-1}\sum_{k=1}^\infty\frac{1}{k^2}\leqslant C'(\gamma)q^{B(3-\gamma)-1}.\end{split}\end{equation}
We want the exponents on $q$ of the right hand side of \eqref{term2}, \eqref{term3} and \eqref{term4} to be negative, that is,
\[\begin{cases}(B+1)(2-\gamma)+\gamma-1<0,\\
\gamma-3<0,\\
(B+1)(3-\gamma)+\gamma-5<0,\\
(B+1)(4-\gamma)+\gamma-7<0,\\
B(3-\gamma)-1<0\end{cases}\]
Therefore, we need 
\begin{equation}\label{condition-negative}\begin{cases}B+1>\frac{1-\gamma}{2-\gamma},\\
2<\gamma_0\leqslant \gamma<3,\\
B+1< \frac{5-\gamma}{3-\gamma},\\
B+1< \frac{7-\gamma}{4-\gamma},\\
B<\frac{1}{3-\gamma}.\end{cases}\end{equation}
Therefore if 
$$\frac{5}{2}<\gamma_0\leqslant \gamma<3,$$
then with the choice of $B$,
$$1<\frac{1-\gamma}{2-\gamma}-1<B<\frac{7-\gamma}{4-\gamma}-1<3,$$
all the inequalities in \eqref{condition-negative} are satisfied, that is, all the exponents on $q$ are negative.
Hence, we have that for each  $\gamma_0\leqslant\gamma<3$, then there exists $\sigma(\gamma)>0$, 
\[\|\mathcal{I}^{\bar q}-\mathbb{I}\|_\gamma\leqslant 0.9+C(\gamma)\frac{1}{\bar q^{\sigma(\gamma)}}\]
So for $\bar q$ large enough, 
$$\|\mathcal{I}^{\bar q}-\mathbb{I}\|_\gamma<0.95,$$
which implies the invertibility of $\mathcal{I}^{\bar q}$ as an operator from $X^\gamma$ to $h^\gamma$. Thus, we finish the proof of Theorem \ref{thm:main}.

\appendix
\section{Higher order constraints}
In this section we derive additional constraints on the infinitesimal deformation function when the smooth deformation preserves both the integrable caustics with rotation numbers $\frac{1}{p}$ and $\frac{2}{p}$ ($p$ is an odd number). The following lemma is a slight modification of the results in \cite[Appendix A and B]{DKW17}.
\begin{lemma} There exists $C=C(\Omega_0)>0$ such that for each odd number $p\geqslant5$, there exist $1$-periodic functions $\gamma_{a,\frac{2k}{p}}(x)$, $\gamma_{b,\frac{2k}{p}}(x)$, $\gamma_{c,\frac{2k}{p}}$, $k=1,\dots,p-1$, such that
$$\|\gamma_{a,\bullet}\|_{C^{r-6}},\|\gamma_{b,\bullet}\|_{C^{r-6}},\|\gamma_{c,\bullet}\|_{C^{r-6}}\leqslant C,$$
and 
for any periodic orbit $x=x_{2,p}^0,\dots, x_{2,p}^{p-1}$ with rotation number $\frac{2}{p}$, we have
\[x_{2,p}^k=x+\frac{2k}{p}+\frac{4\alpha(x+\frac{2k}{p})}{p^2}+\frac{\gamma_{a,\frac{2k}{p}}(x)}{p^4}.\]
Moreover, if $\varphi_{2,p}^k$ denotes the angle of reflection of the trajectory at the $k$-th collision, we have
\[\varphi_{2,p}^k=\frac{2\mu(x_{2,p}^k(x))}{p}\Big(1+\frac{4\beta(x+\frac{2k}{p})}{p^2}+\frac{\gamma_{b,\frac{2k}{p}}(x)}{p^4}\Big),\]
and
\[\frac{\sin\varphi_{2,p}^k(x)}{\mu(x_{2,p}^k(x))}=\frac{2}{p}\Big[1+\frac{4\beta(x+\frac{2k}{p})}{p^2}+S_{2,p}(x+\frac{2k}{p})+\frac{\gamma_{c,\frac{2k}{p}}(x)}{p^4}\Big].\]
Here the function $\alpha$ and $\beta$ are the same as in Lemma \ref{lazutkin}, and 
$$S_{2,p}(x)=\frac{\sin(2\mu(x)/p)}{2\mu(x)/p}-1=\frac{(2\mu(x)/p)^2}{6}\big( -1+O(\frac{1}{p^2})\big).$$
\end{lemma}
Then from Proposition \ref{vanish}, we have
\[\begin{split}0&=\frac{d}{dx}L_{\Gamma_0}^{2,p}\big(n_\Gamma(0,\cdot)\big)(x)=\frac{d}{dx}\sum_{k=0}^{p-1}n_{\Gamma}(0,x_{2,p}^k(x))\mu(x_{2,p}^k(x))\frac{\sin\varphi_{2,p}^k(x)}{\mu(x_{2,p}^k(x))}\\
&=\sum_{k=0}^{p-1}\nu_0'
\left(x+\frac{2k}{p}+\frac{4\alpha(x+\frac{2k}{p})}{p^2}+\frac{\gamma_{a,2k/p}(x)}{p^4}\right)\\
&\quad\quad\quad
\times\Big(1+\frac{4\alpha'(x+\frac{2k}{p})}{p^2}+\frac{\gamma_{a,2k/p}'(x)}{p^4}\Big)\\
&\quad\quad \quad \times \frac{2}{p}\Big[1+\frac{4\beta(x+\frac{2k}{p})}{p^2}+S_{2,p}(x+\frac{2k}{p})+\frac{\gamma_{c,2k/p}(x)}{p^4}\Big]\\
&
\quad+\sum_{k=0}^{p-1}\nu_0\left(x+\frac{2k}{p}+\frac{4\alpha(x+\frac{2k}{p})}{p^2}+\frac{\gamma_{a,2k/p}(x)}{p^4}\right)\\
&\quad\quad\quad
\times\frac{2}{p}\Big[\frac{4\beta'(x+\frac{2k}{p})}{p^2}+S_{2,p}'(x+\frac{2k}{p})+\frac{\gamma'_{c,2k/p}(x)}{p^4}\Big].
\end{split}\]
Now we consider the quantity $$\mathbb{L}_{2,p}\Big(n_\Gamma(0,\cdot)\Big)(x)=\frac{d}{dx}L^{2,p}_{\Gamma_0}\Big(n_{\Gamma}(0,\cdot)\Big)(x)-2\frac{d}{dx}L^{p}_{\Gamma_0}\Big(n_\Gamma(0,\cdot)\Big)(x).$$
Then we have
\[\begin{split}0&=\mathbb{L}_{2,p}\Big(n_\Gamma(0,\cdot)\Big)(x)\\
&=\sum_{k=0}^{p-1}\frac{2}{p}\Big[\big(\nu'_0(x+\frac{2k}{p})-\nu'_0(x+\frac{k}{p})\big)\\
&\quad\quad+\frac{1}{p^2}(4\nu_0'(x+\frac{2k}{p})(\alpha'+\beta-\frac{\mu^2}{6})(x+\frac{2k}{p}))-\nu_0'(x+\frac{k}{p})(\alpha'+\beta-\frac{\mu^2}{6})(x+\frac{k}{p}))\\
&\quad\quad+\frac{1}{p^2}(4\nu_0''(x+\frac{2k}{p})\alpha(x+\frac{2k}{p})-\nu_0''(x+\frac{2k}{p})\alpha(x+\frac{k}{p}))\\
&\quad\quad +\frac{1}{p^2}\Big(4\nu_0(x+\frac{2k}{p})(\beta'(x+\frac{2k}{p})-\frac{[(\mu(x+\frac{2k}{p}))^2]'}{6}\big)\\
&\quad\quad\quad\quad\quad-\nu_0(x+\frac{k}{p})\big(\beta'(x+\frac{k}{p})-\frac{[(\mu(x+\frac{k}{p}))^2]'}{6}\big)\Big)+O(\frac{\|\nu_0'''\|}{p^4})
\end{split}\]
Therefore, we have
\[0=\mathbb{L}_{2,p}\Big(n_\Gamma(0,\cdot)\Big)(x)=\sum_{k=0}^{p-1}\frac{6}{p^3}\Big(\nu_0''\alpha+\nu_0'(\alpha'+\beta-\frac{\mu^2}{6})+\nu_0\beta'-\nu_0(\mu^2)'/6\Big)(x+\frac{k}{p})+O(\frac{\|\nu_0'''\|}{p^4}).\]
Notice that
$$\nu_0''\alpha+\nu_0'\alpha'+\nu_0'\beta-\nu_0'\frac{\mu^2}{6}+\nu_0\beta'-\nu_0(\mu^2)'/6=(\nu_0'\alpha)'+(v_0\beta)'-(v_0\mu^2)'/6$$
Then we obtain a new constraint on the infinitesimal deformation function 
\begin{equation}
0=\sum_{k=0}^{p-1}\frac{6}{p^3}\Big((\nu_0'\alpha)'+(\nu_0\beta)'-(\nu_0\mu^2)'/6\Big)(x+\frac{k}{p})+O(\frac{\|\nu_0'''\|}{p^4}).
\end{equation}
In the same way, for even numbers of the form $q=3\bar{q}+2$, with $\bar{q}$ is also an even number, we could study the   higher order constraints from the preservation of integrable caustics with rotation numbers $\frac{1}{q}$ and $\frac{3}{q}$, and obtain similar restriction on the infinitesimal deformation function. These additional conditions on the infinitesimal deformation function $\mu(x)^{-1}\nu_0$ might lead to further reduction of dimension.

\bigskip 
\bigskip 
\bigskip 
\section{Relation between $\al$, $\beta$, and the radii of the curvature $\rho$}
\label{sec:beta-alpha}
In this section we sketch  the proof of the explicit relation for functions $\alpha$, $\beta$ and the radii of the  curvature of the boundary, as mentioned in Remark \ref{appendix-a-b}. Though this formula is not used in  this work, we hope it would be useful for future research. 
\begin{lemma} \label{lem:alpha-beta}
The following relation between $\alpha(x)$, $\beta(x)$ and the radii of the curvature holds true:
	\[\alpha'(x)=\beta(x)+\frac{-\rho^{1/3}(x)\rho''(x)}{36 C_{\Omega}^2}+\frac{\rho^{-2/3}(x)(\rho'(x))^2}{54C_{\Omega}^2}.\]
	where $C_\Omega=(\int_{\partial\Omega}\rho^{-2/3}(s)ds)^{-1}.$
\end{lemma}
\begin{proof}
	Consider the billiard map in the Lazutkin coordinates (see, e.g. \cite{LAZ73})
	$$x=C_{\Omega}\int_0^s\rho^{-2/3}(s)ds,\quad y=4C_{\Omega}\rho^{1/3}(s)\sin\frac{\varphi}{2},$$
	where $C_{\Omega}^{-1}=\int_{\partial\Omega}\rho^{-2/3}(s)ds.$
	Then the billiard map is written as
	\beaa
	\label{Lazutkin-map}
	f_L : & (x,y)\to (x',y')=
	(x+y+y^3f(x,y),y+y^4g(x,y))\qquad \\
	&=  (x+y+y^3f_0(x)+O(y^4),
	y+y^4g(x,y)),
	\eeaa
	where $f_0(x)=f(x,0)$.

The billiard ball map $(s,\varphi)\mapsto (s',\varphi')$ in Taylor expansion  is (\cite{LAZ73}) 
\[s'=s+a_1\varphi+a_2\varphi^2+a_3\varphi^3+O(\varphi^4),\]
with
\[a_1=2\rho(s),\quad a_2=\frac43\rho(s)\rho'(s),\quad a_3=\frac23\rho^2(s)\rho''(s)+\frac49\rho(s)\rho'^2(s).\]
Through straightforward calculations, in the Lazutkin parametrization, with $x=x(s)$ and $x'=x(s')$we have
\[\begin{split}& x'-x =C_\Omega\int_{s}^{s'}\rho^{-2/3}(\tau)d\tau\\
&=C_\Omega\Big[\rho^{-2/3}[2\rho\varphi+\frac{4}{3}\rho\rho'(\varphi)^2+(\frac{2}{3}\rho^2\rho''+\frac{4}{9}\rho(\rho')^2)(\varphi)^3]\\
&\quad -\frac{1}{2!}\frac{2}{3}\rho^{-5/3}\rho'[2\rho\varphi_{q}^k+\frac{4}{3}\rho\rho'(\varphi_q^k)^2]^2\\
&\quad+\frac{1}{3!}[\frac{10}{9}\rho^{-8/3}(\rho')^2-\frac{2}{3}\rho^{-5/3}\rho''](2\rho\varphi_{q}^k)^3+O(\varphi^4)\Big]\\ 
&=C_\Omega\Big[2\rho^{1/3}\varphi+\Big(\frac{2}{3}\rho^{4/3}\rho''+\frac{4}{9}\rho^{1/3}(\rho')^2-\frac{16}{9}\rho^{1/3}(\rho')^2\\
&\quad \quad +\frac{40}{27}\rho^{1/3}(\rho')^2-\frac{8}{9}\rho^{4/3}\rho''\Big)(\varphi)^3+O((\varphi)^4)\Big]\\
&=C_\Omega\Big[2\rho^{1/3}(s)\varphi+\Big(-\frac29\rho^{4/3}(s)\rho''(s)+\frac4{27}\rho^{1/3}(s)\rho'^2(s)\Big)(\varphi)^3+O((\varphi)^4)\Big].
\end{split}\]
Since $y(s,\varphi)=4C_\Omega\rho^{1/3}(s)\sin \varphi/2$, we have
\[x'-x=y+(\frac{\rho^{-2/3}(x)}{96C_\Omega^2}-\frac{\rho^{1/3}\rho''(x)}{36C_\Omega^2}+\frac{\rho^{-2/3}(x)(\rho'(x))^2}{54C_\Omega^2})y^3+O(y^4).\]
where $\rho(x)$ is read as $\rho(s(x)).$
Therefore we have 
$$f_0(x)=\frac{1}{96C_\Omega^2}\rho^{-2/3}(x)-\frac{\rho^{1/3}\rho''(x)}{36C_\Omega^2}+\frac{\rho^{-2/3}(x)(\rho'(x))^2}{54C_\Omega^2}.$$

      From Lemma \ref{lazutkin}, for the period orbit $(x_q^k,\varphi^k_q)_{k=0,\dots,q-1}$ with rotation number $\frac{1}{q}$, we have that for $k=0,\dots,q-1$, 
	\[\begin{cases}
	x_q^k=x_q^0 + \frac kq + \frac{\al(x_q^0 + \frac kq)}{q^2}+
	O(\frac{1}{q^4}),\\
	\varphi_{q}^k=\frac{\mu(x_q^k)}{q}\big(1+\frac{\beta(x+\frac{k}{q})}{q^2}+O(\frac{1}{q^4})\big ).
	\end{cases}
	\]
	Then
	\[
	x_q^{k+1} - x_q^k= \frac 1q +
	\frac{\al(x_q^0 + \frac{k+1}q)-\al(x_q^0 + \frac kq)}{q^2}+
	O(\frac{1}{q^4})=
	\]
	\[
	=\frac 1q + \frac{\al'(x_q^0 + \frac kq)}{q^3}+
	O(\frac{1}{q^4}).
	\]
	For the $y$-component, recalling that $\mu(x)=\frac{1}{2C_\Omega\rho(x)^{1/3}}$, we have 
		\[
	y_q^{k}=4C_\Omega\rho^{1/3}(x_q^k)\sin\frac{\varphi_q^k}2
	=\frac{1}{q} +  \frac{\beta(x_q^0 + \frac kq)-\frac{1}{24}\mu^2(x_q^0+\frac{k}{q})}{q^3}+
	O(\frac{1}{q^5}).	\]
	Denote \[
	B(x):=\beta(x)-\frac{1}{24}\mu^2(x)=\beta(x)-\frac{1}{96C_\Omega^2\rho^{2/3}(x)}.
	\]
				For the $x$-component, we also have 
	\[
	x_q^{k+1}-x_q^k= y_q^k+ (y_q^k)^3 f_0(x_q^k)+O(|y_q^k|^4)=\frac{1}{q}+\frac{B(x_q^0+\frac{k}{q})+f_0(x_q^0+\frac{k}{q})}{q^3}+O(\frac{1}{q^4}).
	\]
	\[
	=\frac 1q +
	\frac{B(x_q^0 + \frac{k}q)}{q^3}+
	\frac{f_0(x_q^0 + \frac{k}q)}{q^3}+O(\frac{1}{q^4}).
	\]
This leads to the following equality
	\beaa \label{eq:alpha-beta-relation}
	\alpha'(x)=B(x)+f_0(x).
	\eeaa
	That is \[\alpha'(x)=\beta(x)+\frac{-\rho^{1/3}(x)\rho''(x)}{36 C_{\Omega}^2}+\frac{\rho^{-2/3}(x)(\rho'(x))^2}{54C_{\Omega}^2}.\]
	\end{proof}


%

\section*{Acknowledgement}
GH is partially supported by the National Sciences Foundation of China, Grant No. 11790273. Part of this work was done during GH's visit to Department of Mathematics, University of Maryland, College Park, in the Spring of~2018. GH thanks her for the hospitality. VK acknowledges partial support of the NSF grant DMS-1702278. The authors are grateful to H. Hezari for point out an error in an unused formula in the original manuscript.

\end{document}